\documentclass[reqno,10pt]{amsart}

\let\counterwithin\relax
\usepackage{graphicx,amsfonts,amssymb,amsmath,amsthm,url,amscd}
\usepackage{color}
\usepackage{chngcntr}
\usepackage[usenames,dvipsnames]{xcolor}
\usepackage[normalem]{ulem}
\usepackage{graphicx, amsmath, amssymb, amsfonts, amsthm,  tikz, amscd}
\usetikzlibrary{matrix,arrows,decorations.pathmorphing}
\usepackage{graphicx, amsmath, amssymb, amsfonts, amsthm, stmaryrd, tikz, amscd}
\usepackage[all]{xy}
\usetikzlibrary{matrix,arrows,decorations.pathmorphing}

 \usepackage[hyperfootnotes=false, colorlinks, citecolor=	 RoyalBlue, urlcolor=blue, linkcolor=blue         ]{hyperref}

\theoremstyle{plain} 
\newtheorem{theorem}             {Theorem} 

\theoremstyle{definition}

\theoremstyle{plain} 

\theoremstyle{remark}

\counterwithin*{definition}{subsubsection}
\newtheorem*{definition*}  {Definition}

\newtheorem*{example*}    {Example}

\newtheorem{remark}             {Remark}
\newtheorem*{remark*}            {Remark}

\newtheoremstyle{itplain} 
    {6pt}                    
    {5pt\topsep}                    
    {\itshape}                   
    {}                           
    {\itshape}                   
    {.}                          
    {5pt plus 1pt minus 1pt}                       
    {}  

\theoremstyle{itplain} 

\newtheorem*{lemma*}{Lemma}

\newtheorem*{corollary*} {Corollary} 

\theoremstyle{remark} 

\newtheorem*{lemmatest*}{Lemma}

\counterwithin*{lemma}{subsubsection}
\counterwithin*{remark}{subsubsection}
\counterwithin*{corollary}{subsubsection}

\usepackage{etoolbox}
\patchcmd{\section}{\scshape}{\bfseries}{}{}
\makeatletter
\renewcommand{\@secnumfont}{\bfseries}
\makeatother

\renewcommand{\Im}{\mathrm{Im}}

\renewcommand{\geq}{\geqslant}
\renewcommand{\leq}{\leqslant}

\numberwithin{equation}{section}
\DeclareMathOperator{\sgn}{sgn}
\DeclareMathOperator{\SL}{SL}

\DeclareMathOperator{\htt}{ht}

\DeclareMathOperator{\ad}{ad}

\def\eps{\varepsilon}

\def\PGL{\operatorname{PGL}}

\newcommand{\qr}[2]{\left( \frac{#1}{#2} \right)}

\DeclareMathOperator{\trace}{trace}

\DeclareMathOperator{\PSL}{PSL}

\DeclareMathOperator{\nr}{nr}

\DeclareMathOperator{\tr}{tr}

\author{Paul D. Nelson}
\address{ETH Z{\"u}rich, Department of Mathematics, R{\"a}mistrasse 101, CH-8092, Z{\"u}rich, Switzerland}
\email{paul.nelson@math.ethz.ch}

\date{\today}
\title[Bounds for twisted symmetric
square $L$-functions]{Bounds for twisted symmetric
  square $L$-functions via half-integral
  weight periods}
\hypersetup{
  pdfkeywords={},
  pdfsubject={},
  pdfcreator={Emacs 24.5.1 (Org mode 8.2.10)}}
\begin{document}

\begin{abstract}
  We establish the first moment bound
  \[
    \sum_{\varphi}
    L(\varphi \otimes \varphi \otimes \Psi, \tfrac{1}{2}) \ll_\eps p^{5/4+\eps}
  \]
  for triple product $L$-functions, where $\Psi$ is a fixed
  Hecke--Maass form on $\SL_2(\mathbb{Z})$ and $\varphi$ runs
  over the Hecke--Maass newforms on $\Gamma_0(p)$ of bounded
  eigenvalue.  The proof is via the theta correspondence and
  analysis of periods of half-integral weight modular forms.
  This estimate is not expected to be optimal, but the exponent
  $5/4$ is the strongest obtained to date for a moment problem
  of this shape.  We show that the expected upper bound follows
  if one assumes the Ramanujan conjecture in both the integral
  and half-integral weight cases.

  Under the triple product formula, our result may be understood
  as a strong level aspect form of quantum ergodicity:
  for a large prime $p$, all but very few Hecke--Maass newforms
  on $\Gamma_0(p) \backslash \mathbb{H}$ of bounded eigenvalue
  have very uniformly distributed mass after pushforward to
  $\SL_2(\mathbb{Z}) \backslash \mathbb{H}$.

  Our main result turns out to be closely
  related to estimates such as
  \[
    \sum_{|n| < p}
    L(\Psi \otimes \chi_{n p},\tfrac{1}{2}) \ll p,
  \]
  where the sum is over $n$ for which $n p$ is a fundamental
  discriminant and $\chi_{n p}$ denotes the corresponding
  quadratic character.  Such estimates improve upon bounds of
  Duke--Iwaniec.
\end{abstract}

\maketitle

\setcounter{tocdepth}{1} \tableofcontents

\maketitle
\section{Introduction}
\label{sec-1}
\subsection{Overview}
\label{sec-1-1}
The quantum ergodicity theorem
\cite{MR0402834,MR819779,MR916129} says that on a compact Riemannian
manifold with ergodic geodesic flow, almost all eigenfunctions
have equidistributed mass in the large eigenvalue limit.  When
the manifold is arithmetic, additional tools become available by
which one can prove quantitative strengthenings of this
conclusion, to the effect that all but \emph{very few}
eigenfunctions (satisfying additional symmetries) have
\emph{very} equidistributed mass (see \S\ref{sec-1-2} below,
or \cite{MR1361757,luo-sarnak-mass, MR3554896}).
A standard way to
quantify such strengthenings is through upper bounds for
$L^2$-mass variance over families.  Le Masson and Sahlsten
\cite{MR3732880} recently introduced a level aspect variant of
the quantum ergodicity theorem concerning almost all
eigenfunctions in a fixed spectral window on a sequence of
hyperbolic surfaces Benjamini--Schramm-converging to the
hyperbolic plane.  The main results of this article may be
understood as quantitative strengthenings of that result,
for specific classes of eigenfunctions and observables, in the
arithmetic congruence case.

Well-developed techniques for analyzing averages of triple
product $L$-values and/or shifted convolution sums apply in our
setting, giving nontrivial estimates in the intended direction.
We do not apply such techniques here.  We instead introduce
techniques involving the theta correspondence and periods of
half-integral weight modular forms, which seem to give stronger
results in our setting.

\subsection{Context}
\label{sec-1-2}
Let $\mathcal{F}$ traverse a sequence of finite families of cusp
forms defined on congruence covers of the modular surface
$\SL_2(\mathbb{Z}) \backslash \mathbb{H}$
(examples will follow shortly).
For each $\varphi \in \mathcal{F}$,
we may define
a probability measure
$\mu_\varphi$
on $\SL_2(\mathbb{Z}) \backslash \mathbb{H}$
by pushforward of $L^2$-mass:
for $\Psi  \in C_c(\SL_2(\mathbb{Z}) \backslash
\mathbb{H})$,
\[
  \mu_\varphi(\Psi) :=
  \frac{\langle \Psi \varphi, \varphi \rangle}{\langle \varphi, \varphi \rangle},
\]
with $\langle , \rangle$ the Petersson inner product.
For the sequences of families considered in this article,
one either knows or can show readily that
the mean of the $\mu_\varphi$ tends to the uniform measure
as $\mathcal{F}$ varies: for fixed $\Psi$,
\[
  \frac{1}{|\mathcal{F}|} \sum_{\varphi \in \mathcal{F}}
  \mu_\varphi(\Psi)
  =
  \frac{\langle \Psi, 1 \rangle}{\langle 1,1 \rangle}
  + o(1).
\]
We consider here the problem of bounding or estimating the variance sums
\begin{equation}\label{eq:variance-defn}
  V_\mathcal{F}(\Psi)
  :=
  \sum_{\varphi \in \mathcal{F}}
  |\mu_\varphi(\Psi)|^2
\end{equation}
for nice enough fixed
$\Psi : \SL_2(\mathbb{Z}) \backslash \mathbb{H} \rightarrow
\mathbb{C}$ of mean zero.
For concreteness and simplicity,
we suppose throughout this article
that $\Psi$ is a fixed even Hecke--Maass cusp form,
noting that some results quoted below apply to more general
observables than this.

The problem of estimating $V_\mathcal{F}(\Psi)$
becomes more difficult the smaller the family $\mathcal{F}$ is
relative to the parameters of its typical elements.
To illustrate, let $1 \leq H \leq T$.
Let $\mathcal{F}([T,T+H])$ denote the set of
normalized cuspidal Hecke--Maass forms on
$\SL_2(\mathbb{Z}) \backslash \mathbb{H}$ of eigenvalue
$1/4 + t^2$ for some $t \in [T, T + H]$.
Here $H$ varies with $T$ as $T \rightarrow \infty$.
One knows that
$\# \mathcal{F}([T,T+H]) \asymp T H$.
The general quantum ergodicity theorem
implies (for the analogous problem on much more general
manifolds) that
\[
  V_{\mathcal{F}([T,T+H])}
  \ll \frac{T^2}{\log T},
\]
but one expects
the much stronger upper bound
\begin{equation}\label{eqn:expected-LH-spectral}
  V_{\mathcal{F}([T,T+H])}(\Psi)
  \ll_\eps H T^{\eps},
\end{equation}
which should moreover be essentially sharp (i.e., up to the factor
$T^\eps$).
This expectation
is a consequence of the Lindel{\"o}f hypothesis
combined with the triple product formula in the form
\begin{equation}
  |\mu_\varphi(\Psi)|^2
  = T^{-1+o(1)}
  L(\varphi \otimes \varphi \otimes \Psi,\tfrac{1}{2})
\end{equation}
for $\varphi \in \mathcal{F}([T,T+H])$,
where $o(1)$ denotes a quantity tending to zero with $T$.
In the ``long family''
case $H = T$, an asymptotic formula for
$V_{\mathcal{F}([T,2 T])}(\Psi)$ (confirming
a more precise version of \eqref{eqn:expected-LH-spectral})
follows from work of P. Zhao \cite{MR2651907}.
Jung \cite{MR3554896} has confirmed the expectation \eqref{eqn:expected-LH-spectral} for
$H \geq T^{1/3+\eps}$, which appears to be the limit of
current technology.
The upper bound
\begin{equation}\label{eqn:jung-1/3-T-T-1}
  V_{\mathcal{F}([T,T+1])}(\Psi)
  \ll_\eps T^{1/3+\eps}
\end{equation}
obtained from the case $H = T^{1/3+\eps'}$
of \eqref{eqn:expected-LH-spectral}
by positivity likewise appears to be the best
to hope for in the near future.  
Results concerning holomorphic forms entirely analogous to
those quoted above had been obtained earlier in a series of
papers by Luo--Sarnak \cite{MR1361757,luo-sarnak-mass,MR2103474}.

\subsection{Main result}
We pursue here
level aspect
analogues
of the estimate \eqref{eqn:jung-1/3-T-T-1}:
instead of working with increasing eigenvalues
on a fixed surface,
we consider bounded eigenvalues on a tower of congruence covers
of the modular surface.
To that end, fix $\Lambda > 1/4$.
Let $p$ denote a large prime, regarded
as tending to $\infty$.
Let $\mathcal{F}(p)$ denote the set of normalized Hecke--Maass cuspidal newforms $\varphi$ on
$\Gamma_0(p)$ whose Laplace eigenvalue is at most $\Lambda$.
One knows that $\# \mathcal{F}(p) \asymp p$
(see \cite[\S15.5]{MR2061214}),
so the trivial estimate
is
$V_{\mathcal{F}(p)}(\Psi) \ll p$.
The triple product formula, as specialized
in \cite[\S4]{PDN-HQUE-LEVEL},
gives the identity
\begin{equation}\label{eq:ichino-specialized}
  |\mu_\varphi(\Psi)|^2
  =
  p^{-1+o(1)}
  L(\varphi \otimes \varphi \otimes \Psi,\tfrac{1}{2}),
\end{equation}
where $o(1)$ denotes a quantity tending to zero with $p$.
The Lindel{\"o}f hypothesis
thus
suggests
that
\begin{equation}\label{eq:lind-individual}
\mu_{\varphi}(\Psi) \ll p^{-1/2 + o(1)}
\end{equation}
for individual
$\varphi \in \mathcal{F}(p)$,
hence that
\begin{equation}\label{eq:lind-conj-level}
  V_{\mathcal{F}(p)}(\Psi) \ll_\eps p^\eps.
\end{equation}
The nontrivial estimate
$V_{\mathcal{F}(p)}(\Psi) \ll p / \log (p)$ likely follows,
in a more general setting,
from the methods of \cite{MR3732880}
(generalized to non-compact quotients,
and using \cite{MR3742473} to
verify the hypothesis of Benjamini--Schramm convergence).
We establish the following further strengthening:
\begin{theorem}\label{thm:ad}
  Fix an even Hecke--Maass cusp form $\Psi$
  on $\SL_2(\mathbb{Z})$ and $\Lambda > 1/4$.
  Let $\mathcal{F}(p)$ be as above,
  and $V_{\mathcal{F}(p)}$ as in \eqref{eq:variance-defn}.
  Then
  \begin{equation}\label{eq:ad}
    V_{\mathcal{F}(p)}(\Psi)
    \ll_\eps p^{1/4+\eps}.
  \end{equation}
\end{theorem}

By Chebyshev's inequality, we deduce
the following approximation to \eqref{eq:lind-individual}:
\begin{theorem}\label{cor:qqel}
  Fix positive reals $\alpha,\beta$ for which $2 \alpha + \beta < 3/4$.
  Then
  \[
  \# \{\varphi  \in \mathcal{F}(p):
  |\mu_\varphi(\Psi)| > p^{-\alpha}
  \}
  \ll p^{-\beta} \# \mathcal{F}(p).
  \]
\end{theorem}

\subsection{Conditional sharp bounds}
While the Lindel{\"o}f-consistent conjecture
\eqref{eq:lind-conj-level}
appears to be out of reach,
we give a conditional
proof which appears to be the first of its kind.
Here and henceforth let $\chi_d$ denote the quadratic
Dirichlet character attached to a fundamental discriminant $d$.
\begin{theorem}\label{thm:conditional-lind-implies-lind}
  Assume that
  \begin{enumerate}
  \item the Lindel{\"o}f hypothesis
    $L(\Psi \otimes \chi_d, \tfrac{1}{2}) \ll_\eps d^{\eps}$
    holds for the family of quadratic twists of $\Psi$, and
    that
  \item the Ramanujan conjecture
    holds
    for the Hecke eigenvalues of $\Psi$.
  \end{enumerate}
  Then \eqref{eq:lind-conj-level} holds.
\end{theorem}

\subsection{Application to moments of triple product
  $L$-functions}
Under \eqref{eq:ichino-specialized},
Theorems \ref{thm:ad}
and
\ref{thm:conditional-lind-implies-lind} 
translate to moment bounds for the (nonnegative) central values
of some triple product $L$-functions:
\begin{theorem}
  Unconditionally,
  \begin{equation}\label{eq:triple-product-moment-est}
    \sum_{\varphi \in \mathcal{F}(p)}
    L(\varphi \otimes \varphi \otimes \Psi,\tfrac{1}{2})
    \ll p^{5/4+o(1)}.
  \end{equation}
  Under the assumptions of Theorem
  \ref{thm:conditional-lind-implies-lind},
  \begin{equation}\label{eq:lindelof-bound-for-triple-product-open}
    \sum_{\varphi \in \mathcal{F}(p)} L(\varphi \otimes \varphi
    \otimes \Psi,\tfrac{1}{2}) \ll p^{1+o(1)}.
  \end{equation}
\end{theorem}
By comparison, Jung's estimate \eqref{eqn:jung-1/3-T-T-1} translates to
\begin{equation}\label{eq:jung-L-fn}
  \sum_{\varphi \in \mathcal{F}([T,T+1])}
  L(\varphi \otimes \varphi \otimes \Psi,\tfrac{1}{2})
  \ll_\eps T^{4/3+\eps}.
\end{equation}
The first moments on the LHS of \eqref{eq:triple-product-moment-est} and \eqref{eq:jung-L-fn} and
are
analogous in that in each, the analytic
conductor ($\asymp p^4$ and $\asymp T^4$)
is roughly the fourth power of the family cardinality
($\asymp p$ and $\asymp T$).
We note  also that
Iwaniec--Michel \cite{MR1833252}
established (the analogue for holomorphic forms of) the estimate
\begin{equation}\label{eq:}
  \sum_{\varphi \in \mathcal{F}(p)}
  L(\varphi \otimes \varphi,\tfrac{1}{2} )^2
  \ll p^{1 + o(1)},
\end{equation}
which may be understood as a variant of
\eqref{eq:lindelof-bound-for-triple-product-open} in which
$\Psi$ is an Eisenstein series.

\begin{remark}
  A general ``rule of thumb'' in the literature on moment bounds
  for families of $L$-functions is that one should be able to
  establish Lindel{\"o}f-consistent bounds when the family size
  is at least the fourth root of the analytic conductor.  This
  rule does not
  seem to apply when one considers ``sparse''
  moments, such as \eqref{eq:triple-product-moment-est};
  one may understand
  ``sparsity'' here as coming from the coincidence of two of the
  three factors in the triple product $L$-parameter coincide.
  By comparison, it is not difficult to prove
  that
  \[
  \sum_{\varphi \in \mathcal{F}(p)}
  L(\varphi \otimes \Psi_1 \otimes \Psi_2, \tfrac{1}{2})
  \ll p^{1+o(1)}
  \]
  for fixed $\Psi_1, \Psi_2$
  (see \cite{MR2740724} for exact formulas of a similar spirit).  
\end{remark}

\begin{remark}
  We were unable to obtain a ``classical'' proof of
  \eqref{eq:triple-product-moment-est} using the approximate
  functional equation and familiar transformations thereafter.
\end{remark}

\subsection{Application to sparse moments of quadratic twists}
The following curious bound is a byproduct of our method:
\begin{theorem}\label{thm:quad}
  For $C \geq 1$,
  one has
  \begin{equation}
    \sum_{n : |n| \leq C p}
    L(\Psi \otimes \chi_{p n}, \tfrac{1}{2})
    \left( 1 +
      \log (\frac{C p}{|n|})
    \right)
    \ll C p
  \end{equation}
  uniformly in $C,p$,
  where the sum is over integers $n$ for which
  $p n$ is a fundamental discriminant.
  In particular,
  \begin{equation}\label{eq:quad-simple}
    \sum_{n: |n| < p}
    L(\Psi \otimes \chi_{p n}, \tfrac{1}{2})
    \ll p.
  \end{equation}
\end{theorem}
In fact, the methods of this paper
reveal a surprising
relationship between the moments
$\sum_n L(\Psi \otimes \chi_{p n}, \tfrac{1}{2})$ (possibly over
shorter intervals than those above) and
$\sum_{\varphi \in \mathcal{F}(p)} L(\varphi \otimes \varphi
\otimes \Psi,\tfrac{1}{2})$.
This relationship demonstrates the difficulty
underlying an unconditional proof of
\eqref{eq:lind-conj-level}.
We refer to
\S\ref{sec-2} for a detailed discussion of this
relationship,
but record here one consequence:
\begin{theorem}\label{thm:converse-estimate}
  Assume
  \eqref{eq:lind-conj-level}, or equivalently,
  \eqref{eq:lindelof-bound-for-triple-product-open}.
  Assume also that
  $L(\Psi,\tfrac{1}{2}) \neq 0$.
  Then
  \begin{equation}\label{eqn:sparse-quadratic-twist-sum-divided-by-rt-n}
    \sum_{n:|n| < p} \frac{L(\Psi \otimes \chi_{p n}, \tfrac{1}{2})}{|n|^{1/2}} \ll
    p^{1/2 + o(1)}.
  \end{equation}
\end{theorem}
The Lindel{\"o}f hypothesis
implies that the conclusion of Theorem
\ref{eqn:sparse-quadratic-twist-sum-divided-by-rt-n} holds
without the assumption that $L(\Psi,\tfrac{1}{2}) \neq 0$, but
this assumption is necessary for our proof.
The issue is that if
$L(\Psi,\tfrac{1}{2}) = 0$, then
$L(\varphi \otimes \varphi \otimes \Psi,\tfrac{1}{2}) = L(\ad
\varphi \otimes \Psi,\tfrac{1}{2}) L(\Psi,\tfrac{1}{2}) = 0$, so
our hypotheses \eqref{eq:lind-conj-level} and
\eqref{eq:lindelof-bound-for-triple-product-open} hold
trivially,
hence their assumption is not helpful.

We note that the family size in \eqref{eq:quad-simple} is
$\asymp p$, while the analytic conductor in the largest dyadic
range of the sum is $\asymp p^4$, so one might expect the
difficulty of the moment problem addressed by
\eqref{eq:quad-simple}
to be comparable to that for
\begin{equation}\label{eq:psi-chi-n-second-moment}
  \sum_{n:|n| < p} L(\Psi \otimes \chi_n, \tfrac{1}{2})^2,
\end{equation}
\begin{equation}\label{eq:quad-simple-Eis-variant}
  \sum_{n:|n| < p} L(\chi_{p n}, \tfrac{1}{2})^2,
\end{equation}
or
\begin{equation}\label{eq:hb-4th-moment}
  \sum_{n:|n| < p} L(\chi_n, \tfrac{1}{2})^4.
\end{equation}
One may understand \eqref{eq:quad-simple-Eis-variant} as the
variant of \eqref{eq:quad-simple} obtained by taking for $\Psi$
an Eisenstein series.
Heath--Brown
\cite{MR1347489}
proved an upper bound
$p^{1+o(1)}$
for \eqref{eq:hb-4th-moment},
and it seems likely that the same proof
works also for \eqref{eq:quad-simple-Eis-variant}
and \eqref{eq:psi-chi-n-second-moment};
a closely related argument appears implicitly in
\cite{MR1833252}.
We note also that
Soundararajan--Young
\cite{MR2677611}
have established 
an asymptotic formula, conditional on GRH,
for a mild variant of
\eqref{eq:psi-chi-n-second-moment}.

The unconditional estimate \eqref{eq:quad-simple}
established here
seems 
beyond the limits the methods indicated
in the preceeding paragraph:
after
applying an approximate functional equation,
one faces (smooth) sums roughly of the shape
\begin{equation}\label{eq:2}
  S :=
  \sum_{m \sim p^2, n \sim p}
  \frac{\lambda(m)
    \chi_{p n}(m)
  }{\sqrt{m}},
\end{equation}
where
$\lambda(m)$ denotes the $m$th normalized Fourier coefficient of
$\Psi$.
To establish
\eqref{eq:quad-simple}
in this way,
one must show that
$S \ll p^{1+o(1)}$,
which seems hopeless.

We note that the proof of
\eqref{eq:quad-simple}
is specific to the central point
$s=1/2$,
while the proofs indicated above
of analogous estimates for
\eqref{eq:psi-chi-n-second-moment},
\eqref{eq:quad-simple-Eis-variant}
or
\eqref{eq:hb-4th-moment}
apply more generally to $s = 1/2 + it$ for any fixed $t$.

We prove Theorem \ref{thm:quad} at the end of \S\ref{sec-4}.
The proof
goes by the
connection between the values
$L(\Psi \otimes \chi_{p n},\tfrac{1}{2})$ and the squared
magnitudes $|b(p n)|^2$ of the Fourier coefficients of a
half-integral weight lift of $\Psi$;
what we really show is (for instance)
\begin{equation}
  \sum_{n:|n| < p}
|b(p n)|^2 \ll p.
\end{equation}
Such estimates improve in the indicated
range on the diagonal case of
those of
Duke--Iwaniec \cite[(4)]{MR1045402},
which specialize to
(the analogue for holomorphic forms of)
\begin{equation}\label{eq:DI-weaker}
  \sum_{n: |n| < p} |b(p n)|^2
  \ll p^{3/2+\eps}.
\end{equation}

\subsection{Method}
The basic idea behind the proof is to write
\begin{equation}\label{eq:}
  \mu_\varphi(\Psi) = \langle \varphi, G \rangle
\end{equation}
for some automorphic function $G$, and then to estimate the
$L^2$-norm $\langle G, G \rangle$.
The method applies also to
cocompact arithmetic hyperbolic surfaces.
A high level overview is
given in \S\ref{sec-2}.  The overall strategy is related to that
employed in our work on the quantum variance
\cite{nelson-variance-73-2,nelson-variance-II,
  nelson-variance-3} and subconvexity
\cite{nelson-subconvex-reduction-eisenstein} problems, and also
to recent work of Raphael Steiner \cite{2018arXiv181103949S} on
the sup norm problem.

\section{Division of the proof}
\label{sec-2}
The purpose of this section is to reduce
the proof of our main results
to that of some independent claims
to be verified in the body of the paper.

\subsection{Jacobi theta function}
\label{sec:jacobi-thet}
For $z = x + i y \in \mathbb{H}$,
set $\theta(z) := y^{1/4} \sum_{n \in \mathbb{Z}} e(n^2 z)$,
where $e(z) := e^{2 \pi i z}$.
By considering Fourier expansions
at the cusps of $\Gamma_0(4)$,
we obtain the crude upper bound
\begin{equation}\label{eq:theta-height-1-4}
  \theta(z) \ll \htt(z)^{1/4},
\end{equation}
where $\htt(z) := \max_{\gamma \in \SL_2(\mathbb{Z})} \Im(\gamma
z)$.

\subsection{Theta multiplier}
\label{sec-2-1}
For $\gamma \in \Gamma_0(4)$,
set $J(\gamma,z) := \theta(\gamma z) / \theta(z)$.
We recall from \cite[Prop 2.2]{MR0332663} that if $\gamma = \begin{pmatrix}
  a & b \\
  c & d
\end{pmatrix}$,
then
\begin{equation}\label{eq:formula-for-metaplectic-cocycle}
  J(\gamma,z)
  =
  \eps(\gamma)
  \frac{\sqrt{c z + d}}{|c z + d|^{1/2}},
  \quad
  \eps(\gamma)
  =
  \eps_d^{-1}
  \qr{c}{d}.
\end{equation}
Here we define the square root
by $\sqrt{z} := |z|^{1/2} e^{i \arg (z)/2}$
if $z = |z| e^{i \arg(z)}$
with $\arg(z) \in (-\pi,\pi]$,
$\eps_d = 1$ or $i$ according as $d \equiv 1$ or $-1$ modulo $4$,
and
$\qr{c}{d}$ is the ``quadratic residue symbol'' characterized
in \cite[p442]{MR0332663}.
\subsection{Petersson inner product}
\label{sec-2-2}
For a congruence subgroup $\Gamma \leq \Gamma_0(4)$ and $\kappa
\in \mathbb{Z}$,
we call
$F : \mathbb{H} \rightarrow \mathbb{C}$
\emph{modular of weight $\kappa/2$ on $\Gamma$}
if $F(\gamma z) = J(\gamma,z)^\kappa F(z)$ for $\gamma \in
\Gamma$.
If $F_1,F_2$ are modular of weight $\kappa/2$ on $\Gamma$,
then the function $F_1 \overline{F_2}$ is $\Gamma$-invariant;
if it induces a function on $\Gamma \backslash \mathbb{H}$
that is integrable
with respect to the measure $d \mu(z) := \frac{d x \, d
  y}{y^2}$,
then we define the normalized Petersson inner product
\[
\langle
F_1, F_2 \rangle
:=
\frac{1}{[\PSL_2(\mathbb{Z}) : \overline{\Gamma }]}
\int_{z \in \Gamma \backslash \mathbb{H}}
F_1(z) \overline{F_2(z)} \, d \mu(z)
\]
and associated norm $\|F\| := \langle F, F \rangle^{1/2}$;
here $\overline{\Gamma} \leq \PSL_2(\mathbb{Z})$ denotes the
image of $\Gamma$.
Note that the definition of $\langle,\rangle$ is invariant
under shrinking $\Gamma$.
\subsection{The half-integral weight lift and its Fourier expansion}
\label{sec-2-3}
We recall in \S\ref{sec-3-3-2}
the construction via theta lifting of a Maass/inverse-Shimura/Shintani lift
$h$ of $\Psi$.
It is nonzero precisely when $L(\Psi,\tfrac{1}{2}) \neq 0$.
It is modular of weight $1/2$ on $\Gamma_0(4)$,
and belongs to an analogue of the Kohnen plus space.
It admits a Fourier expansion
\[
h(z) = \sum_{n \in \mathbb{Z}_{\neq 0}}
\frac{b(n)}{|n|^{1/2}}
W(n y)
e(n x)
\]
where $W$ is a Whittaker function (see \S\ref{sec-3-3-3})
and $b(n) = 0$ unless $n \equiv 0,1 \pod{4}$.
When $d$ is a fundamental discriminant,
we have by \cite[(5.17)]{MR3549627}
\begin{equation}\label{eq:b-d-in-terms-of-L}
  |b(d)|^2
= c L(\Psi \otimes \chi_d, \tfrac{1}{2}).
\end{equation}
Here $c$ depends only upon $\Psi$,
and $c \neq 0$ whenever $h \neq 0$.
More generally, any $n \equiv 0,1\pod{4}$
may be written uniquely
as $n = d \delta^2$,
where $d$ is a fundamental discriminant and $\delta$ is a
natural number;
we may then deduce via the Shimura relation
the estimate (see  \cite[Prop 6.1]{2016arXiv160604119L})
\[|b(n)| \ll_{\eps} |b(d)| \delta^{\vartheta+\eps},\]
where $\vartheta \in [0,7/64]$
quantifies the known bounds towards the Ramanujan conjecture for
the Hecke eigenvalues of $\Psi$.
Conrey--Iwaniec \cite{MR1779567} have shown that
$L(\Psi \otimes \chi_d, \tfrac{1}{2}) \ll_\eps d^{1/3+\eps}$.
Since $\vartheta \leq 1/3$,
it follows in general that
\begin{equation}\label{eq:bound-for-b-of-n}
  |b(n)|^2 \ll_\eps |n|^{1/3+\eps}.
\end{equation}
On the other hand, it is expected
(by the Lindel{\"o}f hypothesis
for $L(\Psi \otimes \chi_d, \tfrac{1}{2})$
and the Ramanujan conjecture for $\Psi$)
that
\begin{equation}\label{eq:bound-for-b-of-n-conditional}
  |b(n)|^2 \ll_\eps |n|^{\eps}.
\end{equation}
\subsection{Application of an incomplete Hecke operator}
\label{sec-2-4}
We denote by $h^\sharp$
the (normalized) application
to $h$ of a variant of the classical ``$U_p$'' operator:
\begin{equation}\label{eq:defn-h-sharp}
  h^\sharp(z)
  :=
  \frac{1}{p^{1/2}}
  \sum_{j \in \mathbb{Z}/p \mathbb{Z} }
  h (\frac{z + p j }{p^2})
  =
  \sum_{n \in \mathbb{Z}_{\neq 0}}
  \frac{b(p n)}{|n|^{1/2}}
  W(\frac{n y}{p})
  e (\frac{n x}{p}).
\end{equation}
We record in \S\ref{sec:fourier-exp-h-sharp-other-cusps}
the level of $h^\sharp$.
We show in \S\ref{sec-4} that
\begin{equation}\label{eqn:l-2-bound-for-h-sharp}
  \langle h^\sharp, h^\sharp \rangle = \langle h, h \rangle,
\end{equation}
which reflects a special feature of $1/2$-integral weight forms.
\subsection{Properties of the varying forms}
\label{sec-2-5}
We assume that each $\varphi \in \mathcal{F}(p)$
is arithmetically normalized, so that
\[
\varphi(z)
=
\sum_{n \in \mathbb{Z}_{\neq 0}}
\frac{\lambda_\varphi(|n|)}{|n|^{1/2}}
W_{\varphi}(n y) e(n x),
\]
where $\lambda_{\varphi}(1) = 1$
and $W_{\varphi}(y) = 2 y^{1/2} K_{i t_{\varphi}}(2 \pi y)$ for $y > 0$,
where $1/4 + t_\varphi^2$ is the Laplace eigenvalue of $\varphi$.
Each $\varphi \in \mathcal{F}(p)$
is an eigenfunction
of the Atkin--Lehner/Fricke involution with eigenvalue $\pm 1$,
that is to say,
$\varphi(-1/(p z)) = \pm \varphi(z)$,
or equivalently,
\begin{equation}\label{eq:fricke-involution-applied-to-phi}
  \varphi(-1/z)
  = \pm \varphi(z/p),
\end{equation}
and it is known \cite{HL94} that
\begin{equation}\label{eq:HL}
  p^{-\eps} \ll_\eps \langle \varphi, \varphi  \rangle \ll_{\eps} p^{\eps}.
\end{equation}
\subsection{An explicit seesaw identity}
\label{sec-2-6}
We show in \S\ref{sec-5}
that for $\varphi \in \mathcal{F}(p)$,
\begin{equation}\label{eq:period-formula}
  \mu_{\varphi}(\Psi)
  = \pm 4 \langle \varphi (\tfrac{4 z}{p}) \vartheta(z), h^\sharp(z) \rangle,
\end{equation}
where $\pm$ is as in \eqref{eq:fricke-involution-applied-to-phi}.
Here we abuse notation mildly by writing
simply
(e.g.)
$\varphi (\tfrac{4 z}{p})$
for the function
$z \mapsto \varphi (\tfrac{4 z}{p})$.
A notable feature of the RHS of \eqref{eq:period-formula}
is that it depends linearly upon $\varphi$.

\begin{remark*}
  Related identities involving forms of level $1$ have been
  given by Biro \cite{MR2783966}; the proof given here is
  different and applies more generally (e.g., also to compact
  arithmetic quotients $\Gamma \backslash \mathbb{H}$).  See
  also Ichino \cite[\S11]{MR2198222}.  Similar identities
  are also implicit in our work \cite{nelson-variance-73-2,
    nelson-variance-II,nelson-variance-3} on the quantum variance problem.
\end{remark*}

\subsection{Reduction to period bounds}
\label{sec-2-7}
By \eqref{eq:HL}, \eqref{eq:period-formula}
and Bessel's inequality,
we have
\begin{equation}\label{eq:post-bessel}
  \sum_{\varphi \in \mathcal{F}(p)}
|\mu_\varphi(\Psi)|^2
=
16
\sum_{\varphi \in \mathcal{F}(p)}
\frac{|\langle
  \varphi (\tfrac{4 z}{p}),
  \overline{\theta(z)}
  h^\sharp(z)
  \rangle|^2}{\langle \varphi, \varphi  \rangle}
\ll
p^{o(1)}
\langle \overline{\theta} h^\sharp,
\overline{\theta} h^\sharp
\rangle,
\end{equation}
so the proof of Theorem \ref{thm:ad}
reduces to that of the estimate
\begin{equation}\label{eq:required-bound-theta-h-sharp}
  \|\theta h^\sharp\|^2
  \ll p^{1/4+o(1)}.
\end{equation}

\begin{remark*}
  This part of the argument is reminiscent of arguments in
  \cite{MR2972602, MR780071, MR2726097} and (implicitly) in
  \cite{venkatesh-2005, michel-2009}, among other places.
  A
  more general (but less elementary) approach to identities like \eqref{eq:post-bessel}
  is given in \cite{nelson-subconvex-reduction-eisenstein},
  following \cite{MR3291638}.
\end{remark*}

\subsection{The basic inequality\label{sec:baisc-ineq-summary}}
\label{sec-2-8}
We show in \S\ref{sec-6} that for any $T \geq 1$,
\begin{equation}\label{eq:flexible-bound-theta-h-sharp}
\|\theta h^\sharp \|^2 \ll T^{1/2} + p^{-1/2} R,
\end{equation}
where
\[
R :=
\sum_{n}
\frac{|b(p n)|^2}{|n|^{1/2}}
\left( 1 + \frac{|n|}{p/T}
\right)^{-100}.
\]
This is a key step in the argument, so we sketch here the basic
idea behind the proof.  We estimate separately the
contributions to $\|\theta h^\sharp\|^2$ from the ranges
$\{z : \htt(z) \leq T\}$ and $\{z : \htt(z) > T\}$, where $\htt$
is as in \S\ref{sec:jacobi-thet}.  In both ranges we apply the
pointwise bound \eqref{eq:theta-height-1-4} for $\theta$.  For
the range where $\htt(z) \leq T$, we apply the $L^2$-bound
\eqref{eqn:l-2-bound-for-h-sharp} for $h^\sharp$ to obtain the
estimate $\ll T^{1/2}$.  For the range where $\htt(z) > T$, we
apply Parseval to the Fourier expansion of $h^\sharp$
and appeal to the rapid decay of the Whittaker function $W$ to
establish estimates such as
\begin{equation}\label{eqn:}
  \frac{1}{p} \int_{y=T}^{\infty} y^{1/2} \int_{x=0}^p
  |h^\sharp(x+iy)|^2 \, \frac{d x \, d y}{y^2} \ll p^{-1/2} R.
\end{equation}

\begin{remark*}
  It would be possible to refine the present
  analysis by  employing the spectral expansion of $|\theta|^2$
  as in \cite{nelson-theta-squared}
  in place of the upper bound \eqref{eq:theta-height-1-4},
  but doing so does not seem
  to lead to stronger unconditional results.
\end{remark*}

\subsection{Completion of the proof}
\label{sec-2-9}
By taking $T = p^{1/2}$
in \eqref{eq:flexible-bound-theta-h-sharp}
and appealing to the Conrey--Iwaniec bound
\eqref{eq:bound-for-b-of-n},
we readily obtain \eqref{eq:required-bound-theta-h-sharp}, hence Theorem \ref{thm:ad}.
For the proof of Theorem \ref{thm:conditional-lind-implies-lind},
we note that its hypotheses imply \eqref{eq:bound-for-b-of-n-conditional},
which gives the required bound upon
taking $T = 1$ 
in \eqref{eq:flexible-bound-theta-h-sharp}.

The proof of Theorem \ref{thm:converse-estimate}
is recorded in \S\ref{sec:converse-estimate}.

\begin{remark}
  Taking $T = p^{1+\eps}$ in \eqref{eq:flexible-bound-theta-h-sharp}
  and appealing to
  any convexity bound of the form $b(p n) \ll (p n)^{O(1)}$
  already gives
  the nontrivial estimate
  $\sum_{\varphi \in \mathcal{F}(p)}
  |\mu_\varphi(\Psi)|^2 \ll p^{1/2+o(1)}$.
\end{remark}
\begin{remark}
  It is natural to ask the questions:
  \begin{enumerate}
  \item Is the Conrey--Iwanec bound an essential input to the
    method?
  \item Can one do better by exploiting the average over $n$ in
    $R$?
  \end{enumerate}
  To address these, let $R^\flat$ denote the subsum over $R$ obtained by
  restricting to summation indices $n$ for which $p n$ is a
  fundamental discriminant.
  Then
  \[
  R^\flat
  =
  c
  \sum_{
    \substack{
      n : \\
      p n \text{ is fundamental}
    }
  }
  \frac{L(\Psi \otimes \chi_{p n},\tfrac{1}{2})}{|n|^{1/2}}
  \left( 1 + \frac{|n|}{p/T} \right)^{-100}
  \]
  for some $c$ depending only upon $\Psi$.
  It seems likely that
  \begin{enumerate}
  \item one can show directly using an
    approximate functional equation \cite[\S5.2]{MR2061214}
    and Heath--Brown's large sieve
    for quadratic characters \cite{MR1347489}
    that for $1 \leq T \leq p$, one has
    $R^{\flat} \ll T^{-1/2} p^{1 + o(1)}$,
    and that
  \item one can establish the same bound for $R$ by using the
    Rankin--Selberg upper bound
    $\sum_{n \leq x} |\lambda_{\Psi}(n)|^2 \ll x$ to control the
    contribution from non-fundamental discriminants.
  \end{enumerate}
  If so, then by taking $T = p^{1/2}$ one obtains a proof of
  Theorem \ref{thm:ad} that does not rely upon the
  Conrey--Iwaniec bound.  Conversely, we do not know how to do
  better by exploiting the average over $n$ in $R$.
\end{remark}
\section{Preliminaries}
\label{sec-3}
\subsection{Generalities}
\label{sec-3-1}
\subsubsection{}
\label{sec-3-1-1}
Let $B,C$ be coprime natural numbers.
Set
\[
\Gamma_0(C/B)
:=
\left\{ \begin{pmatrix}
    a & b \\
    c & d
  \end{pmatrix} \in \SL_2(\mathbb{Z}) :
  B \mid b,
  C \mid c
\right\}.
\]
When $B = 1$, this is the standard definition of $\Gamma_0(C) = \Gamma_0(C/1)$.
In general, $\Gamma_0(C/B)$ is a congruence
subgroup of $\SL_2(\mathbb{Z})$ that is conjugate to $\Gamma_0(B C)$.

As motivation for the notation,
note that if $F : \mathbb{H} \rightarrow \mathbb{C}$ is
$\SL_2(\mathbb{Z})$-invariant,
then the function $z \mapsto F(z C/B)$ is $\Gamma_0(C/B)$-invariant.
\subsubsection{}
\label{sec-3-1-2}
Let $R := M_2(\mathbb{Z})$ denote the ring of $2 \times 2$
integral matrices.
Set $S := \mathbb{Z} + 2 R$.
We use a superscripted $0$ to denote ``traceless'' elements,
so that for instance,
\[
M_2(\mathbb{R})^0
= \left\{ \begin{pmatrix}
    a & b \\
    c & -a
  \end{pmatrix} : a,b,c \in \mathbb{R} \right\},
\]
\[
S^0 = \left\{ \begin{pmatrix}
    a & 2 b \\
    2 c & -a
  \end{pmatrix} : a,b,c \in \mathbb{Z}  \right\}.
\]
Note that
\begin{equation}\label{eq:decomp-S-Z-S0}
  S = \mathbb{Z} \oplus S^0
\end{equation}

\subsubsection{}
\label{sec-3-1-3}
For natural numbers $B,C$,
set
\[
R(C/B)
:=
\left\{ \begin{pmatrix}
    a & b \\
    c & d
  \end{pmatrix} \in M_2(\mathbb{Q}) :
  a,d \in \mathbb{Z};
  b \in B^{-1} \mathbb{Z},
  c \in C \mathbb{Z}
\right\}.
\]
It is a lattice in $M_2(\mathbb{Q})$.
We abbreviate $R(C/1) := R(C)$.
We note that $R(C/B)$ is not directly related to $\Gamma_0(C/B)$
except when $B = 1$,
in which case $\Gamma_0(C) = \SL_2(\mathbb{Z}) \cap R(C)$.
The significance of the notation is that
$R(C/B)$
and
$R(B/C)$ are dual lattices with respect
to the quadratic form on $M_2(\mathbb{Q})$
defined by the determinant.

When $B,C$ are odd,
we set $S(C/B) := S \cap R(C/B)
= \mathbb{Z} + 2 R(C/B)$ and $S^0(C/B) := S^0 \cap S(C/B)$.
\subsubsection{}
\label{sec-3-1-4}
For $w = u + i v \in \mathbb{H}$,
define $\sigma_w \in \SL_2(\mathbb{R})$
by the formula
\[
\sigma_w := \begin{pmatrix}
  v^{1/2} & u v^{-1/2} \\
   & v^{-1/2}
\end{pmatrix},
\]
so that $\sigma_w i = w$.
\subsection{Theta kernels}
\label{sec-3-2}
We recall the definitions and basic properties of some theta
kernels.
We refer to \cite{MR0389772}, \cite[\S2]{watson-2008}
and \cite[App. B]{MR3356036} for details.

In what follows,
take $w, w_1,w_2, z = x + iy \in \mathbb{H}$.
\subsubsection{}
\label{sec-3-2-1}
Define
$P : M_2(\mathbb{R}) \rightarrow \mathbb{R}$
by
\[
P (\begin{pmatrix}
  a & b \\
  c & d
\end{pmatrix})
:=
\frac{a^2 + b^2 + c^2 + d^2}{2},
\]
$\phi_{w,z}^0 : M_2(\mathbb{R})^0 \rightarrow \mathbb{C}$
by
\[
\phi_{w,z}^0(\alpha)
:=
\frac{1}{2 \pi}
y^{3/4}
\exp(- 2 \pi y P(\sigma_w^{-1} \alpha \sigma_w))
e(x \det(\alpha)),
\]
and
$\phi_{w_1,w_2,z} : M_2(\mathbb{R}) \rightarrow
\mathbb{C}$
by
\[
\phi_{w_1, w_2,z}(\alpha)
:=
\frac{1}{2 \pi}
y
\exp(- 2 \pi y P(\sigma_{w_1}^{-1} \alpha \sigma_{w_2}))
e(x \det(\alpha)).
\]

Note that for $\alpha = m + \beta$ with $m \in \mathbb{R}, \beta
\in M_2(\mathbb{R})^0$,
\begin{equation}\label{eq:factorization-schwartz-at-infinity}
  \phi_{w,w,z}(\alpha)
  = y^{1/4} e(m^2 z) \phi_{w,z}^0(\beta).
\end{equation}

\subsubsection{}
\label{sec-3-2-2}
Set
\[
\theta(w,z)
:=
\sum_{\alpha \in S^0}
\phi_{w,z}^0(\alpha).
\]
Then $\theta$ defines a theta kernel
({\`a} la Maass--Shintani--Waldspurger),
of weight $-1/2$ on $\Gamma_0(4)$ in the variable $w$
and of weight $0$ on $\SL_2(\mathbb{Z})$ in the variable $z$
(cf. \cite{MR0389772}, \cite[\S2]{MR1244668}).

\subsubsection{}
\label{sec-3-2-3}
For a lattice
$L \subseteq M_2(\mathbb{Q})$,
set
\[
\theta(L;w_1,w_2,z)
:=
\sum_{\alpha \in L}
\phi_{w_1,w_2,z}(\alpha).
\]
This defines a modular function
of weight $0$ in each variable
with respect to suitable congruence subgroups
(see \cite{MR0389772}).

\subsubsection{}
The lattice $R(p)$ has discriminant $p^2$
and dual $R(1/p)$.
The quadratic form $(R(p),\det)$ has signature $(2,2)$.
Thus (see \cite{MR0389772})
\begin{equation}\label{eq:aut-rels-for-theta-4}
  \theta(R(p);w_1,w_2,-1/z)
  = p^{-1}
  \theta(R(1/p); w_1, w_2, z).
\end{equation}

\subsection{Ternary theta lifts}
\label{sec-3-3}
\subsubsection{\label{sec:Psi-Fourier-exp}}
\label{sec-3-3-1}
We assume that
$\Psi : \SL_2(\mathbb{Z}) \backslash \mathbb{H} \rightarrow \mathbb{C}$
is arithmetically normalized,
so that its Fourier expansion reads
\[
\Psi(w)
=
\sum_{n \in \mathbb{Z}_{\neq 0}}
\frac{\lambda_{\Psi}(n)}{|n|^{1/2}}
W_{\Psi}(n y) e(n x)
\]
where $\lambda_{\Psi}(n) = \lambda_{\Psi}(|n|)$
satisfies
$\lambda_{\Psi}(1) = 1$
and so that
the Ramanujan conjecture
reads $|\lambda_{\Psi}(n)| \leq \sum_{d \mid n} 1$,
while $W_{\Psi}(y) = 2 |y|^{1/2} K_{i t_\Psi}(2 \pi |y|)$,
with $1/4 + t_\Psi^2$ the Laplace eigenvalue of $\Psi$.

\subsubsection{}
\label{sec-3-3-2}
Define $h : \mathbb{H} \rightarrow \mathbb{C}$
by requiring that
\[
\overline{h(z)}
:=
\int_{w \in \SL_2(\mathbb{Z}) \backslash \mathbb{H}}
\theta(w,z)
\Psi(w) \, d \mu(w).
\]
Then $h$ is a constant multiple
of the form constructed
in
\cite[Prop 2.3]{MR1244668},
and has the properties indicated in \S\ref{sec-2-3}
(compare with \cite[\S5]{MR3549627} and \cite[\S6]{2016arXiv160604119L}).

\subsubsection{}
\label{sec-3-3-3}
The Whittaker function $W$ of $h$ is given by
$W(y) = W_{\sgn(y) /4, i t_\Psi/2} (4 \pi |y|)$.  Since $\Psi$
is non-constant, we have $t_\Psi \in \mathbb{R}$ or
$i t_\Psi \in (-1/2,0) \cup (0,1/2)$.
(The second possibility does not actually
occur,
thanks to the known Selberg eigenvalue conjecture
for $\PGL_2(\mathbb{Z})$,
but we do not need to exclude it.)
We require the estimate
(see \cite[\S7.3]{MR3096570}, \cite[(0.11)]{MR1244668}, \cite[13.14.21]{NIST:DLMF})
\begin{equation}\label{eq:whittaker-fn-estimate}
  W(y) \ll
  \min(|y|^{1/2 - \vartheta/2}, |y|^{1/4} e^{- 2 \pi |y|}).
\end{equation}
Here the implied constant and
$\vartheta \in (|\Im(t_\psi)|, 1/2)$ depend upon $t_\Psi$, which
is fixed for us.
In particular, $W(y) \ll |y|^{1/4+\delta}$ for some fixed $\delta > 0$.
\subsubsection{\label{sec:expand-h-other-cusps}}
\label{sec-3-3-4}
A set of inequivalent cusps for $\Gamma_0(4)$
is given by
$\{\infty, 0, 1/2\}$.
It is shown in \cite[\S11]{2016arXiv160604119L} that
\begin{equation}\label{eq:h-FE-at-0}
  e^{\pi i /4} (z/|z|)^{-1/2} h(-1/4z) = \sqrt{2}
\sum_{n \equiv 0(4)}
\frac{b(n)}{|n|^{1/2}}
W (\frac{n y}{4})
e (\frac{n x}{4})
\end{equation}
and
\begin{equation}\label{eq:h-FE-at-1-2}
e^{\pi i /4} (z/|z|)^{-1/2}
h (\frac{1}{2} - \frac{1}{4 z})
= \sqrt{2}
\sum_{n \equiv 1(4)}
\frac{b(n)}{|n|^{1/2}}
W (\frac{n y}{4})
e (\frac{n x}{4}).
\end{equation}
Thus the expansion of $h$ at any cusp of $\Gamma_0(4)$ is
obtained from that at the cusp $\infty$ essentially by
restricting the summation index to some congruence class modulo
$4$.  (Strictly speaking, since $h$ has
  half-integral weight, ``the'' expansion of $h$ at a cusp
  depends upon the
  choice of matrix to represent the cusp.
  We hope this abuse of
  terminology introduces no confusion.)

For $\ell = 1,2$,
set
\[
h_\ell(z) :=
\left( \frac{\ell z + 1}{|\ell z + 1|} \right)^{-1/2}
h (\begin{pmatrix}
  1 &  \\
  \ell & 1
\end{pmatrix} z).
\]
Using the readily-verified identities
\[
  h (\begin{pmatrix}
  1 &  \\
  1 & 1
\end{pmatrix} z)
= h(\frac{-1}{z+1}),
\quad 
  h (\begin{pmatrix}
  1 &  \\
  2 & 1
\end{pmatrix} z)
=
h
(
\frac{1}{2} - \frac{1/2}{2 z + 1}
),
\]
we may read off the Fourier expansions of $h_1$ and $h_2$
at the cusp $\infty$
from
\eqref{eq:h-FE-at-0}
and
\eqref{eq:h-FE-at-1-2}.
\subsection{Cusps of $\Gamma_0(4/p)$}
\label{sec-3-4}
For general background on cusps and fundamental domains
we refer to \cite{MR0314766, Iw97, MR1942691}.

\subsubsection{}
\label{sec-3-4-1}
For the remainder of \S\ref{sec-3-4},
we set $\Gamma := \Gamma_0(4/p)$.
We note that $-1 \in \Gamma$.

\subsubsection{}
\label{sec-3-4-2}
Set
$\Delta := \left\{ \pm \begin{pmatrix}
    1 & n \\
    & 1
  \end{pmatrix} : n \in \mathbb{Z}  \right\} \leq
\SL_2(\mathbb{Z})$.
Then any set $\mathcal{C} := \{\gamma_1,\dotsc,\gamma_6\}
\subseteq \SL_2(\mathbb{Z})$
consisting of elements of the form
\[
\gamma_1 :=
\begin{pmatrix}
  1 &  \\
  & 1
\end{pmatrix},
\quad
\gamma_2 :=
\begin{pmatrix}
  1 &  \\
  2 & 1
\end{pmatrix},
\quad
\gamma_3 :=
\begin{pmatrix}
  1 &  \\
  1 & 1
\end{pmatrix},
\]
\[
\gamma_4 := \begin{pmatrix}
  p & \ast \\
  4 & \ast
\end{pmatrix},
\quad
\gamma_5 := \begin{pmatrix}
  p & \ast \\
  2 & \ast
\end{pmatrix},
\quad
\gamma_6 := \begin{pmatrix}
  p & \ast \\
  1 & \ast
\end{pmatrix}
\]
gives representatives for the double coset space
$\Gamma \backslash \SL_2(\mathbb{Z}) / \Delta$,
and $\{\gamma_j \infty : j =1, \dotsc, 6\} = \{\infty, 1/2, 1, p/4, p/2, p\}$
is a maximal set of inequivalent cusps
for $\Gamma \backslash \mathbb{H}$.
\subsubsection{}
\label{sec-3-4-3}
For $\gamma \in \SL_2(\mathbb{Z})$, the width of the cusp
$\gamma \infty$
for $\Gamma \backslash \mathbb{H}$
is the cardinality $w(\gamma)$ of the preimage in $\Gamma
\backslash \SL_2(\mathbb{Z})$
of $\Gamma \gamma \Delta$.
We have
\[
w(\gamma_j) = p,p,4 p, 1, 1, 4 \text{ for } j=1, \dotsc, 6, \text{ respectively.}
\]
\subsubsection{\label{sec:fund-domain-thin-part-Gamma}}
\label{sec-3-4-4}
Recall that $\htt : \Gamma \backslash \mathbb{H} \rightarrow
\mathbb{C}$
is defined
by $\htt(z) := \max_{\gamma \in \SL_2(\mathbb{Z})} \Im(\gamma
z)$.
By tiling $\Gamma \backslash \mathbb{H}$
by translates of the standard fundamental domain for
$\SL_2(\mathbb{Z})$,
we see that $\htt(z) \geq \sqrt{3}/2$ for all $z$.
Moreover, for any $T > 1$,
the union
\begin{equation}\label{eq:fund-domain-cusp-part-Gamma}
  \cup_{\gamma \in \mathcal{C}}
  \{\gamma(x + i y) : 0 \leq x \leq w(\gamma), y \geq T\}
\end{equation}
is essentially disjoint and
gives a fundamental domain for
$\{z \in \Gamma \backslash \mathbb{H} : \htt(z) \geq T\}$.
Finally,
for given $y_0 \in (0,1)$,
the fibers
of the natural map
\[
\{x + i y : 0 \leq x \leq p, y \geq y_0\}
\rightarrow \Gamma \backslash \mathbb{H}
\]
have cardinality $O(1/y_0)$, uniformly in $p$ \cite[Lem 2.10]{MR1942691}.

\subsubsection{\label{sec:fourier-exp-h-sharp-other-cusps}}
\label{sec-3-4-5}
Using \cite[Prop 1.3, Prop 1.5]{MR0332663} (or more precisely
their analogue for Maass forms),
we verify that $h^\sharp$ is modular
(of weight $1/2$)
on $\Gamma$.
Using
\S\ref{sec:expand-h-other-cusps},
we see that
the Fourier expansion of $h^\sharp$ at the cusps $1/2, 1$ of
$\Gamma$ is obtained
from its expansion \eqref{eq:defn-h-sharp} at $\infty$ 
essentially by restricting the summation index
to a congruence class modulo $4$.
More precisely, for $\ell = 1, 2$,
define
$h_\ell^\sharp$
in terms of $h_\ell$ analogously to how
$h^\sharp$ was defined in terms of $h$.
\begin{lemma*}
  We have
\begin{equation}
  \left( \frac{\ell z + 1}{|\ell z + 1|} \right)^{-1/2}
  h^\sharp (\begin{pmatrix}
    1 &  \\
    \ell & 1
  \end{pmatrix} z)
  =
  h_\ell^\sharp(z).
\end{equation}
\end{lemma*}
\begin{proof}
  Set $w := \begin{pmatrix}
  1 &  \\
  \ell & 1
\end{pmatrix} z$.
It suffices to show for each $j \in \mathbb{Z}/p \mathbb{Z}$ that
\[
\left( \frac{\ell z + 1}{|\ell z + 1|} \right)^{-1/2} h (\frac{w
  + p j}{p^2}) = h_\ell (\frac{z + p j}{p^2}).
\]
We may assume that $j$ is represented
by a negative integer divisible by $4$.
The conclusion follows then
in a straightforward manner from the second assertion
of \cite[Lem 3.4]{MR0332663}
with $N := 4, M := 1, K := 4$,
and
\[
\begin{pmatrix}
  a & b \\
  c & d
\end{pmatrix}
:=
\begin{pmatrix}
  p^{-1} & j \\
  & p
\end{pmatrix}
\begin{pmatrix}
  1 &  \\
  \ell & 1
\end{pmatrix}
\begin{pmatrix}
  p^{-1} & j \\
  & p
\end{pmatrix}^{-1}
=
\begin{pmatrix}
  1  + j p \ell &  - j^2 \ell \\
  p^2 \ell & 1 - j p \ell
\end{pmatrix}.
\]
\end{proof}

\subsection{Explicit Shimizu lifts}
\label{sec-3-5}
We make use of the following explicit form
of the Shimizu correspondence.
\begin{lemma*}
  Let $\varphi \in \mathcal{F}(p)$.
  Then for $w_1,w_2,z \in \mathbb{H}$,
  \begin{equation}
    \overline{\varphi(w_1)} \varphi(w_2)
    = \int_{z \in \Gamma_0(p) \backslash \mathbb{H}}
    \theta(R(p);w_1,w_2,z)
    \varphi(z)
    \, d \mu(z).
  \end{equation}
\end{lemma*}
\begin{proof}
  \cite[Thm 5.2]{MR3356036} implies an analogous
  assertion for holomorphic forms.  Running through the proof of
  that theorem with ``$k := 0$'', we obtain the identity stated
  here.
\end{proof}

\section{The $L^2$-norm of $h^\sharp$}
\label{sec-4}
In this section we establish \eqref{eqn:l-2-bound-for-h-sharp}.
We open
$\langle h^\sharp, h^\sharp \rangle$
as a double sum over $j_1, j_2$.
The diagonal $j_1 = j_2$ contributes $\langle h,h \rangle$;
our task is thus to show that
\begin{equation}\label{eq:sweet-cancellation-task-1}
  \sum _{\substack{
      j_1, j_2 \in \mathbb{Z} / p  : \\
      j_1 \neq j_2
    }
  }
  \langle
  h (\frac{z + p j_1 }{p^2}),
  h (\frac{z + p j_2 }{p^2})
  \rangle = 0,
\end{equation}
where here and henceforth
we mildly abuse notation
by writing (e.g.) $h (\frac{z + p j_i }{p^2})$
for the function
$z \mapsto h (\frac{z + p j_i }{p^2})$.
By a simple change of variables,
our task \eqref{eq:sweet-cancellation-task-1}
reduces to verifying that
\begin{equation}\label{eq:sweet-cancellation-task-2}
  \sum_{j=1}^{p-1}
  \langle h(z),
  h(z + j/p) \rangle
  = 0.
\end{equation}

To that end, temporarily fix $j \in \{1,\dotsc,p-1\}$.
Choose
$\gamma =
\begin{pmatrix}
  a & \ast \\
  p c & d
\end{pmatrix} \in \Gamma_0(4)$
with $c j \equiv -1 \pod{p}$, $c > 0$
and
$d \equiv 1 \pod{4 p}$,
so that also $a \equiv 1 \pod{4 p}$.
Set
\[n(j/p) := \begin{pmatrix}
  1 & j/p \\
  & 1
\end{pmatrix},
\quad 
t(p^{-1})
:=
\begin{pmatrix}
  p^{-1} &  \\
  & p
\end{pmatrix}
\]
and
$\delta := n(j/p) \gamma t(p^{-1})$.
By direct calculation,
we see that
$\delta$ belongs to $\Gamma_0(4)$ and has the form $\begin{pmatrix}
  \ast & \ast \\
  c & p d
\end{pmatrix}$.
By the invariance of the Petersson inner product and the modularity of $h$, we deduce that
\begin{align*}
  \langle h(z), h(z + j/p) \rangle
  &=
    \langle h(z), h(n(j/p) z) \rangle
  \\
  &=
    \langle h(\gamma z), h(n(j/p) \gamma z) \rangle
  \\
  &=
    \langle h(\gamma z), h(\delta p^{2} z) \rangle
  \\
  &=
    \eta(j)
    \langle h(z), h(p^{2} z) \rangle
\end{align*}
where
\begin{align*}
  \eta(j) &:=
  \exp(\tfrac{i}{2}
  (
  \underbrace{\arg( p c z + d) - \arg (c p^2 z + p d)}_{=0}
  ))
  \eps(\gamma) \overline{\eps(\delta)} \\
          &=
            \eps_{d}^{-1} \eps_{p d}
\qr{pc}{d}
\qr{c}{p d}
=
\eps_p
\qr{-j}{p}.
\end{align*}
In the final step we invoked our assumptions on $c,d$
and the rules of \cite[p442]{MR0332663}.
Thus $(\mathbb{Z}/p \mathbb{Z})^\times \ni j \mapsto \eta(j)$ defines a
constant multiple of the nontrivial quadratic character, hence
its sum over $j=1,\dotsc,p-1$ vanishes.
This completes the proof of
\eqref{eq:sweet-cancellation-task-2},
hence of
\eqref{eq:sweet-cancellation-task-1},
hence of \eqref{eqn:l-2-bound-for-h-sharp}.
\begin{remark*}
  The ``trivial bound'' for
  $\langle h^\sharp, h^\sharp \rangle$, ignoring the
  cancellation derived above from the oscillation of the
  half-integral weight automorphy factor,
  is $O(p^{\vartheta})$.
\end{remark*}

To illustrate the surprising power of
\eqref{eqn:l-2-bound-for-h-sharp}, we now prove Theorem
\ref{thm:quad}.
Define $V_0 : \mathbb{R}_{>0} \rightarrow \mathbb{R}_{\geq 0}$
by
\[
V_0(u)
:=
\int_{y = u}^{\infty}
|W(y)|^2 \, \frac{d y}{y^2}.
\]
By the asymptotic expansion of $W$ near $0$, there is $u_0 \in (0,1/e)$
so that $V_0(u) \asymp \log(1/u)$ for $|u| < u_0$.
Let $C \geq 1$.
By \eqref{eqn:l-2-bound-for-h-sharp},
the final assertion of \S\ref{sec-3-4-4}, Parseval
and the change of variables $y \mapsto p y / n$, we obtain
\begin{align*}
  1
  &\gg
    \|h^\sharp\|^2 \\
  &\gg
    \frac{1}{C p}
    \int_{x=0}^p
    \int_{y=u_0/C}^{\infty}
    |h^\sharp(x+iy)|^2 \, \frac{d x \, d y}{y^2}
  \\
  &=
    \frac{1}{C}
    \sum_n
    \frac{|b(p n)|^2}{|n|}
    \int_{y = u_0/C}^{\infty}
    |
    W(\frac{n y}{p})|^2
    \, \frac{d y}{y^2}
  \\
  &=
    \frac{1}{C p}
    \sum_{n}
    |b(p n)|^2
    V_0 (u_0 \frac{n}{C p})
  \\
  &\gg
    \frac{1}{C p}
    \sum_{n:|n| \leq C p}
    |b(p n)|^2
    \log
    (u_0^{-1} \frac{C p}{n})
  \\
  &\geq
    \frac{1}{C p}
    \sum_{n:|n| \leq C p}
    |b(p n)|^2
    (1 + 
    \log(\frac{C p}{n})).
\end{align*}
We conclude by the $L$-value formula
\eqref{eq:b-d-in-terms-of-L}
for $|b(d)|^2$.

\section{A triple product identity}
\label{sec-5}
In this section we prove \eqref{eq:period-formula}
after developing some preliminaries.
\subsection{Linear algebra lemma}
\label{sec-5-1}\label{sec:lin-alg-ident}
In this subsection we
employ
the temporary notation
$G := \SL_2(\mathbb{F}_p)$,
$M = M_2(\mathbb{F}_p)$.
The group $G$ acts on the $\mathbb{F}_p$-vector space
$M$ by conjugation.
Let $L := \left\{ \begin{pmatrix}
    0 & \ast \\
    0 & 0
  \end{pmatrix} \right\}$ denote the strictly upper-triangular
subspace of $M$.
Let $B$ denote the upper-triangular subgroup
of $G$.
Intrinsically, $B$ is the normalizer in $G$ of $L$.
We note, for future reference,
that $G$, $B$ and $L$ are the images under reduction modulo $p$
of $\SL_2(\mathbb{Z})$, $\Gamma_0(p)$ and $p S(1/p)$, respectively.
We require the following calculation:
\begin{lemma*}
  For $x \in M$,
  we have
  \begin{equation}\label{eqn:linear-algebra-lemma}
      \# \{g \in G/B : x \in g L g^{-1}\}
  =
  \begin{cases}
    p+1 & \text{if } x = 0, \\
    1 & \text{if } x \neq 0, \tr(x) = \det(x) = 0 \\
    0 & \text{otherwise.}
  \end{cases}
  \end{equation}
\end{lemma*}
\begin{proof}
  Let $e_1 = \begin{pmatrix}
    1  \\
    0
  \end{pmatrix}$ and $e_2 = \begin{pmatrix}
    0  \\
    1
  \end{pmatrix}$ denote the standard basis elements for the
  two-dimensional $\mathbb{F}_p$-vector space $\mathbb{F}_p^2$.
  The subspace $L$ of $M$ consists of those nilpotent elements
  $x \in M$ that annihilate the line $\ell := \mathbb{F}_p e_1$
  spanned by $e_1$.  The conjugate subspace $g L g^{-1}$ thus
  consists of the nilpotent elements of $M$ that annihilate the
  translated line $g \ell$.  As $g$ traverses $G/B$, the
  translated line $g \ell$ traverses the set of all lines in
  $\mathbb{F}_p^2$.  Thus the LHS of
  \eqref{eqn:linear-algebra-lemma} vanishes unless $x$ is
  nilpotent, in which case it is the number of lines that $x$
  annihilates.

  In verifying \eqref{eqn:linear-algebra-lemma}, we may and
  shall assume that $x$ is nilpotent, since otherwise both sides
  vanish.
  The claimed identity is then an immediate
  consequence of the following observations:
  \begin{itemize}
  \item The zero element annihilates
    all $p+1$ lines in
    $\mathbb{F}_p^2$.
  \item Any nilpotent element has nontrivial kernel, hence
    annihilates some line.
  \item No nonzero nilpotent element annihilates two distinct
    lines.
  \end{itemize}
\end{proof}

\subsection{Pushforward of the theta kernel}
\label{sec-5-2}\label{sec:theta-kernel-pushforward}
For $w,z \in \mathbb{H}$, set
\[
\theta^\sharp(w,z)
:=
\sum_{\gamma \in \Gamma_0(p) \backslash \SL_2(\mathbb{Z})}
\theta(S(1/p), \gamma w, \gamma w, z).
\]
\begin{lemma*}
  We have
  $\theta^\sharp(w,z)
  = p \theta(z)
  \left( \frac{1}{p^{1/2}}
    \sum_{j \in \mathbb{Z}/p \mathbb{Z}}
    \theta (w, \tfrac{z + p j}{p^2})
    + \theta(w,z)
  \right)$.
\end{lemma*}
\begin{proof}
  Abbreviate $\phi := \phi_{w,w,z}, \phi^0 := \phi_{w,z}^0$.
  By definition,
  \[
  \theta^\sharp(w,z)
  = \sum_{\gamma \in \Gamma_0(p) \backslash \SL_2(\mathbb{Z})}
  \sum_{\alpha \in \gamma^{-1} S (1/p) \gamma}
  \phi(\alpha),
  \quad
  \theta(w,z)
  = \sum_{\beta \in S^0}
  \phi^0(\beta).
  \]
  For each $\gamma \in \SL_2(\mathbb{Z})$,
  we have $S \subseteq \gamma^{-1} S (1/p ) \gamma \subseteq p^{-1}
  S$.
  Conversely, let $\alpha \in p^{-1} S$.
  By applying the lemma of \S\ref{sec:lin-alg-ident}
  to the image of $p \alpha$ under reduction modulo $p$,
  we see that
  $\alpha$ belongs to $\gamma^{-1} S (1/p) \gamma$
  for some $\gamma \in \SL_2(\mathbb{Z})$
  if and only if $\trace(\alpha) \in \mathbb{Z}$
  and $\det(\alpha) \in p^{-1} \mathbb{Z}$,
  and in that case,
  \[
  \# \{
  \gamma \in \Gamma_0(p) \backslash \SL_2(\mathbb{Z})
  :
  \alpha \in \gamma^{-1} S(1/p) \gamma 
  \}
  =
  \begin{cases}
    1 &  \text{if } \alpha \notin S,
    \\
    p+1 & \text{if } \alpha \in S.
  \end{cases}
  \]
  Thus
  \[
  \theta^\sharp(w, z)
  =
  \sum_{
    \substack{
      \alpha \in p^{-1} S : \\
      \tr(\alpha) \in \mathbb{Z}, \\
      \det(\alpha) \in p^{-1} \mathbb{Z}
    }
  }
  \phi(\alpha)
  + p
  \sum_{\alpha \in S}
  \phi(\alpha).
  \]
  Using the decomposition \eqref{eq:decomp-S-Z-S0}
  and the identity \eqref{eq:factorization-schwartz-at-infinity},
  we obtain
  \[
  \sum_{\alpha \in S}
  \phi(\alpha)
  =
  \sum_{m \in \mathbb{Z}, \beta \in S^0}
  y^{1/4} e(m^2 z)
  \phi^0(\beta)
  =
  \theta(z) \theta(w,z),
  \]
  \[
  \sum_
  {
    \substack{
      \alpha \in p^{-1} S : \\
      \tr(\alpha) \in \mathbb{Z}, \\
      \det(\alpha) \in p^{-1} \mathbb{Z}
    }
  }
  \phi(\alpha)
  =
  \theta(z)
  \sum _{\substack{
      \beta \in p^{-1} S^0 :  \\
      \det(\beta) \in p^{-1} \mathbb{Z} 
    }
  }
  \phi^0(\beta),
  \]
  \[
  \sum _{\substack{
      \beta \in p^{-1} S^0 :  \\
      \det(\beta) \in p^{-1} \mathbb{Z} 
    }
  }
  \phi^0(\beta)
  =
  \sum_{\substack{
      \beta \in S^0 :  \\
      \det(\beta) \in p \mathbb{Z} 
    }
  }
  \phi^0(p^{-1} \beta)
  =
  p^{1/2}
  \sum_{j \in \mathbb{Z}/p \mathbb{Z}}
  \theta(w, \frac{z + p j }{p^2}).
  \]
  These identities combine to yield the required identity.
\end{proof}

\subsection{Period identities\label{sec:pf-thm-period-formula}}
\label{sec-5-3}
We now prove \eqref{eq:period-formula}.
Note first that we may explicitly pushforward
the $L^2$-mass of $\varphi$
down to $\SL_2(\mathbb{Z}) \backslash \mathbb{H}$
before we integrate it against $\Psi$:
\begin{equation}\label{eq:we-can-pushforward}
  \langle \Psi \varphi, \varphi  \rangle
  =
  \frac{1}{p+1}
  \int_{w \in \SL_2(\mathbb{Z}) \backslash \mathbb{H}}
  \Psi(w)
  \sum_{\gamma \in \Gamma_0(p) \backslash \SL_2(\mathbb{Z})}
  |\varphi|^2(\gamma w)
  \, d \mu(w).
\end{equation}
This identity motivates finding a formula
for the inner sum over $\gamma$.
To that end,
we begin by applying
the substitution
$z \mapsto -1/z$
to the integral representation for $\overline{\varphi}(w_1)
\varphi(w_2)$
given by the lemma of \S\ref{sec-3-5}.
Using
the formulas
\eqref{eq:fricke-involution-applied-to-phi}
and
\eqref{eq:aut-rels-for-theta-4}
describing the behavior of
the integrand under $z \mapsto -1/z$,
we obtain
\begin{equation}\label{eqn:identity-phi-w1-w2-R-one-over-p}
  \overline{\varphi(w_1)} \varphi(w_2)
= \frac{\pm  1}{p}
\int_{z \in \Gamma_0(1/{p}) \backslash \mathbb{H}}
\theta(R(1/p); w_1,w_2,z)
\varphi(\tfrac{z}{p})
\, d \mu(z).
\end{equation}
Next,
we observe that
$R(1/p) = \{\alpha/2 : \alpha \in S(1/p), \nr(\alpha) \equiv 0
\pod{4}\}$.
From this,
the identity
\[
  \phi_{w_1,w_2,(z+ p j)/4}(\alpha) = (1/4) e(p j
  \det(\alpha)/4) \phi_{w_1,w_2,z}(\alpha/2)
\]
and
finite Fourier inversion,
we deduce that
\begin{equation}\label{eq:}
  \theta(R(1/p); w_1, w_2, z)
  =
  \sum_{j \in \mathbb{Z}/4 \mathbb{Z}}
  \theta(S(1/p); w_1, w_2,
  \tfrac{z + p j}{4}).
\end{equation}
We choose a congruence subgroup $\Gamma$ of $\Gamma_0(1/p)$
that contains $\{\pm 1\}$
but is otherwise small enough that each integral
displayed below is well-defined.
Since
$\varphi(\tfrac{z - p j}{p}) = \varphi(\tfrac{z}{p})$, the
substitution $z \mapsto z - p j$ followed by $z \mapsto 4 z$
gives
\begin{align*}
  &
    \int_{z \in \Gamma \backslash \mathbb{H}}
    \theta(S(1/p); w_1, w_2,
    \tfrac{z + p j}{4})
    \varphi(\tfrac{z}{p})
    \, d \mu(z)
  \\
  &\quad
    =
    \int_{z \in \Gamma \backslash \mathbb{H}}
    \theta(S(1/p); w_1, w_2,
    \tfrac{z}{4})
    \varphi(\tfrac{z}{p})
    \, d \mu(z)
  \\
  &\quad
    =
    \int_{z \in \Gamma \backslash \mathbb{H}}
    \theta(S(1/p); w_1, w_2, z)
    \varphi(\tfrac{4 z}{p})
    \, d \mu(z)
\end{align*}
We introduce the shorthand
$\mathbb{E}_{z \in \Gamma \backslash \mathbb{H}} (\dotsb)$
for $[\SL_2(\mathbb{Z}) : \Gamma]^{-1} \int_{z \in \Gamma \backslash
  \mathbb{H}}
(\dotsb)$.
Using the above computations,
the identity $[\SL_2(\mathbb{Z}) : \Gamma_0(1/p)] = p+1$
and the formula \eqref{eqn:identity-phi-w1-w2-R-one-over-p},
we deduce that
\begin{equation}\label{eq:final-general-formula-varphi-w1-w2}
  \overline{\varphi(w_1)} \varphi(w_2)
  =
  \frac{\pm 4 (p+1)}{p}
  \mathbb{E} _{z \in \Gamma \backslash \mathbb{H}}
  \theta(S(1/p); w_1, w_2, z)
  \varphi(\tfrac{4 z}{p})
  \, d \mu(z).
\end{equation}
Setting $w_1 = w_2 =: w$
in \eqref{eq:final-general-formula-varphi-w1-w2}
gives
\[
\sum_{\gamma \in \Gamma_0(p) \backslash \SL_2(\mathbb{Z})}
|\varphi|^2(\gamma w)
=
  \frac{\pm 4 (p+1)}{p}
\mathbb{E} _{z \in \Gamma \backslash \mathbb{H}}
\theta^\sharp(w,z)
\varphi(\tfrac{4 z}{p})
\, d \mu(z).
\]
Applying the lemma of \S\ref{sec:theta-kernel-pushforward}, we obtain
\begin{equation}\label{eq:formula-for-varphi-2-pushforwarded}
  \sum_{\gamma \in \Gamma_0(p) \backslash \SL_2(\mathbb{Z})}
  |\varphi|^2(\gamma w)
  =
  \pm 4 (p+1)
  \mathbb{E} _{z \in \Gamma \backslash \mathbb{H}}
  \theta(z)
  \frac{1}{p^{1/2}}
  \sum_{j \in \mathbb{Z}/p \mathbb{Z} }
  \theta(w, \tfrac{z + p j}{p^2})
  \varphi(\tfrac{4 z}{p})
  \, d \mu(z).
\end{equation}
(We have used here that $\theta(z) \theta(w, z)$ is old at $p$
and $\varphi$ is new at $p$ to discard the contribution
of the second term
in the
lemma of \S\ref{sec:theta-kernel-pushforward}.)
We note that
\begin{equation}\label{eq:h-sharp-as-theta-lift}
  \int_{w \in \SL_2(\mathbb{Z}) \backslash \mathbb{H}}
  \Psi(w)
  \frac{1}{p^{1/2}}
  \sum_{j \in \mathbb{Z}/p \mathbb{Z} }
  \theta(w, \tfrac{z + p j}{p^2}) \, d \mu(w)
  =
  \overline{h^\sharp(w)}.
\end{equation}
Combining \eqref{eq:we-can-pushforward} and \eqref{eq:h-sharp-as-theta-lift}
with
\eqref{eq:formula-for-varphi-2-pushforwarded}
integrated over $w$ against $\Psi(w)$
gives
\begin{align*}
  \langle \Psi \varphi, \varphi  \rangle
  &=
    \pm 4
    \mathbb{E}_{z \in \Gamma \backslash \mathbb{H}}
    \theta(z)
    \overline{h^\sharp(z)}
    \varphi(\tfrac{4 z}{p})
    \, d \mu(z)
  =
    \pm 4
    \langle \theta(z)
    \varphi(\tfrac{4 z}{p}),
    h^\sharp(z)
    \rangle,
\end{align*}
as required.

\section{The basic inequality}
\label{sec-6}
We now prove \eqref{eq:flexible-bound-theta-h-sharp}.
Set $\Gamma := \Gamma_0(4/p)$.
Define $F : \Gamma \backslash \mathbb{H} \rightarrow
\mathbb{R}_{\geq 0}$
by
$F(z) := \htt(z)^{1/2} | h^\sharp(z)|^2$.
By \eqref{eq:theta-height-1-4}
and the estimate $[\SL_2(\mathbb{Z}) : \Gamma ] \asymp p$,
our task  \eqref{eq:flexible-bound-theta-h-sharp}
reduces to showing for $T \geq 1$ that
\begin{equation}\label{eq:flexible-bound-for-F}
  \int_{\Gamma \backslash \mathbb{H}}
  F :=
  \int_{z \in \Gamma \backslash \mathbb{H}}
  F(z) \, d \mu(z)
  \ll
  p
  T^{1/2}
  +
  p^{1/2}
  R,
\end{equation}
where $R$ is as in \S\ref{sec:baisc-ineq-summary}
and the implied constant is uniform in $p,T$.
Using the fundamental domain from
\S\ref{sec:fund-domain-thin-part-Gamma},
we may write
$\int_{\Gamma \backslash \mathbb{H}} F
= I_0 + \sum_{\gamma \in \mathcal{C}} I(\gamma)$,
where
\[
I_0 
:=
\int_{z \in \Gamma \backslash \mathbb{H} : \htt(z) \leq T}
F(z) \, d \mu(z)
= \int _{
  \substack{
    z \in \Gamma \backslash \mathbb{H} : \\
    \htt(z) \leq T
  }
}
\htt(z)^{1/2}
| h^\sharp(z) |^2
\, d \mu(z),
\]
\begin{align*}
I(\gamma) &:=
\int_{y=T}^{\infty} 
\int_{x = 0}^{w(\gamma)}
F(\gamma(x+iy))
\, 
\frac{d x \, d y}{y^{2}} \\
&=
\int_{y=T}^{\infty} 
y^{1/2}
\int_{x = 0}^{w(\gamma)}
|h^\sharp(\gamma(x+iy))|^2
\, 
\frac{d x \, d y}{y^{2}}.
\end{align*}

The adequate estimate $I_0 \ll p T^{1/2}$ follows as indicated in
\S\ref{sec:baisc-ineq-summary} from the estimate
$\|h^\sharp\|^2 = \|h\|^2 \ll 1$.

It remains to estimate $I(\gamma_i)$ for $i=1,\dotsc,6$.
We start with the most important case $i=1$.
Substituting the formula
\eqref{eq:defn-h-sharp} for $h^\sharp$ and appealing to
Parseval followed by the substitution
$y \mapsto p y/n$,
we obtain
\begin{align}\label{eq:formula-I-gamma-1-yay}
  I(\gamma_1)
  &=
    p
    \int_{y=T}^{\infty} 
    y^{1/2}
    \sum_n
    \frac{|b(p n)|^2}{|n|}
    | W (\frac{n}{p}y) |^2
    \, 
    \frac{d x \, d y}{y^{2}}
  \\ \label{eq:formula-V}
  &= 
    p^{1/2}
    \sum_n
    \frac{|b(p n)|^2}{|n|^{1/2}}
    V(\frac{n}{p/T}), 
    \quad
    V(u)
    :=
    \int_{y = u}^{\infty}
    y^{1/2}
    | W (y) |^2
    \, \frac{d y}{y^2}.
\end{align}
The estimate \eqref{eq:whittaker-fn-estimate} for $W$
implies that
$V(u) \ll \min(1, |u|^{-100})$,
which leads to the adequate estimate
$I(\gamma_1) \ll p^{1/2} R$.

Using \S\ref{sec:fourier-exp-h-sharp-other-cusps},
we may similarly estimate $I(\gamma_2), I(\gamma_3)$.

Since $w(\gamma_i) \leq 4 = O(1)$
for $i=4,5,6$,
the ``trivial bound''
$I(\gamma_i) \ll w(\gamma_i) \|h^\sharp\|_{\infty}^2 \ll p
w(\gamma_i)$
suffices for our purposes.

\section{A converse estimate}\label{sec:converse-estimate}
We finally prove Theorem \ref{thm:converse-estimate}.
The non-constant Fourier components of $\theta$
decay rapidly near the cusp, so we may find some fixed $y_0 \geq 1$ so that
\begin{equation}\label{eqn:}
  \theta(z) \gg y^{1/4} \text{ if } y = \Im(z) \geq y_0.
\end{equation}
Assuming \eqref{eq:lind-conj-level} or
\eqref{eq:lindelof-bound-for-triple-product-open} and arguing as
in \S\ref{sec-6}, we derive the lower bound
\begin{align}
  p^{1 + o(1)}
  &\gg
  p \|\theta h^\sharp\|^2
  \gg
  \int_{y=y_0}^{\infty} 
  y^{1/2}
  \int_{x = 0}^{p}
  |h^\sharp(\gamma(x+iy))|^2
  \, 
    \frac{d x \, d y}{y^{2}}
  \\
  &=
    p^{1/2}
    \sum _{n}
    \frac{|b(p n)|^2}{|n|^{1/2}}
    V (y_0 n/p),
\end{align}
with $V$ as defined in \eqref{eq:formula-V}.  Assuming that
$L(\Psi,\tfrac{1}{2}) \neq 0$, the form $h$ is nonzero, and its
Whittaker function $W$ is not identically zero on any interval.
In particular, $V(u) \gg 1$ for $u \leq y_0$.  The required
estimate \eqref{eqn:sparse-quadratic-twist-sum-divided-by-rt-n}
follows.

\subsection*{Acknowledgements}
We gratefully acknowledge the support of NSF grant OISE-1064866
and SNF grant SNF-137488 during the work leading to this paper.
Most of this article was written while the author was in
residence at the Mathematical Sciences Research Institute in
Berkeley, California, during the Spring 2017 semester, supported
by the National Science Foundation under Grant No. DMS-1440140.
We would like to thank the anonymous referee for helpful
feedback and corrections and to thank Valentin Blomer, Philippe
Michel, Etienne Le Masson, Maksym Radziwill, K. Soundararajan,
Raphael Steiner and Matthew Young for helpful feedback and
encouragement.

\def\cprime{$'$} \def\cprime{$'$} \def\cprime{$'$} \def\cprime{$'$}

\end{document}